\documentclass[twoside,10.5pt]{article}
\usepackage{amsfonts}
\usepackage{mathrsfs}
\usepackage{pifont}
\usepackage{amsmath}
\usepackage{amsthm}
\usepackage{txfonts}
\usepackage{geometry}
\usepackage{latexsym}
\usepackage{amssymb}
\usepackage{graphicx}
\usepackage{geometry}
\usepackage{xcolor}

\input{tcilatex}

\begin{document}

\title{\parbox{\linewidth}{\footnotesize\noindent } {\LARGE Hamiltonian
vector fields on Weil bundles}}
\author{Norbert\ MAHOUNGOU\ MOUKALA\thanks{{\footnotesize nmahomouk@yahoo.fr}%
}, Basile Guy Richard BOSSOTO\thanks{{\footnotesize bossotob@yahoo.fr}} \and %
Marien NGOUABI University, Brazzaville, Congo}
\date{}
\maketitle

\begin{abstract}
Let $M$ be a paracompact smooth manifold, $A$ a Weil algebra and $M^{A}$ the
associated Weil bundle. In this paper, we give a characterization of
hamiltonian field on $M^{A}$ in the case of Poisson manifold and of
Symplectic manifold.
\end{abstract}



\section{Introduction}

In what follows, we denote $M$, a paracompact differentiable manifold of
dimension $n$, $C^{\infty }(M)$ the algebra of smooth functions on $M$ and $A
$ a local algebra in the sense of Andr\'{e} Weil i.e a real commutative
algebra of finite dimension, with unit, and with an unique maximal ideal $%
\mathfrak{m}$ of codimension $1$ over $%
\mathbb{R}
$\cite{wei}.. In this case, there exists an integer $h$ such that $\mathfrak{%
m}^{h+1}=(0)$ and $\mathfrak{m}^{h}\neq (0)$. The integer $h$ is the height
of $A$. Also we have $A=%
\mathbb{R}
\oplus \mathfrak{m}$.

\bigskip We recall that a near point of $x\in M$ of kind $A$ is a morphism
of algebras%
\begin{equation*}
\xi :C^{\infty }(M)\longrightarrow A
\end{equation*}%
such that 
\begin{equation*}
\xi (f)-f(x)\in \mathfrak{m}
\end{equation*}%
for any $f\in C^{\infty }(M)$. We denote $M_{x}^{A}$ the set of near points
of $x\in M$ of kind $A$ and $M^{A}=\bigcup\limits_{x\in M}M_{x}^{A}$ the
manifold of infinitely near points on $M$ of kind $A$ and 
\begin{equation*}
\pi _{M}:M^{A}\longrightarrow M
\end{equation*}%
the projection which assigns every infinitely near point to $x\in M$ to its
origin $x$. The triplet $(M^{A},\pi _{M},M)$ defines a bundle called bundle
of infinitely near points or simply Weil bundle\cite{kms}..\newline

\bigskip When $M$ and $N$ are smooth manifolds and when $h:M\longrightarrow N
$ is a differentiable map of class $C^{\infty },$ then the map 
\begin{equation*}
h^{A}:M^{A}\longrightarrow N^{A},\xi \longmapsto h^{A}(\xi )
\end{equation*}%
such that for all $g$ in $C^{\infty }(N)$, 
\begin{equation*}
\lbrack h^{A}(\xi )](g)=\xi (g\circ h)
\end{equation*}%
is differentiable,\cite{mor}.. Thus, for $f\in C^{\infty }(M),$ the map 
\begin{equation*}
f^{A}:M^{A}\longrightarrow \mathbb{R}^{A}=A,\xi \longmapsto \lbrack
f^{A}(\xi )](id_{\mathbb{R}})=\xi (id_{\mathbb{R}}\circ f)=\xi (f)
\end{equation*}%
is differentiable of class $C^{\infty }$. The set, $C^{\infty }(M^{A},A)$ of
smooth functions on $M^{A}$ with values on $A,$ is a commutative algebra
over $A$ with unit and the map 
\begin{equation*}
C^{\infty }(M)\longrightarrow C^{\infty }(M^{A},A),f\longmapsto f^{A}
\end{equation*}%
is an injective morphism of algebras. Then, we have \cite{bok2}.: 
\begin{equation}
(f+g)^{A}=f^{A}+g^{A};(\lambda \cdot f)^{A}=\lambda \cdot f^{A};(f\cdot
g)^{A}=f^{A}\cdot g^{A}.  \notag
\end{equation}

In \cite{bok2}. and \cite{nbo}., we showed that the following assertions are
equivalent:

\begin{enumerate}
\item A vector field on $M^{A}$ is a differentiable section of the tangent
bundle $(TM^{A},\pi _{M^{A}},M^{A})$.

\item A vector field on $M^{A}$ is a derivation of $C^{\infty }(M^{A})$.

\item A vector field on $M^{A}$ is a derivation of $C^{\infty }(M^{A},A)$
which is $A$-linear.

\item A vector field on $M^{A}$ is a linear map $X:C^{\infty
}(M)\longrightarrow C^{\infty }(M^{A},A)$ such that 
\begin{equation*}
X(f\cdot g)=X(f)\cdot g^{A}+f^{A}\cdot X(g),\quad \text{for any}\,f,g\in
C^{\infty }(M)\text{.}
\end{equation*}
\end{enumerate}

\bigskip In all what follows, we denote $\mathfrak{X}(M^{A})$ the set of
vector fields on $M^{A}$ and $Der_{A}[C^{\infty }(M^{A},A)]$ the set of $A$%
-linear maps%
\begin{equation*}
X:C^{\infty }(M^{A},A)\longrightarrow C^{\infty }(M^{A},A)
\end{equation*}%
such that%
\begin{equation*}
X(\varphi \cdot \psi )=X(\varphi )\cdot \psi +\varphi \cdot X(\psi ),\quad 
\text{for any}\,\varphi ,\psi \in C^{\infty }(M^{A},A)\text{.}
\end{equation*}

Then \cite{nbo}.,%
\begin{equation*}
\mathfrak{X}(M^{A})=Der_{A}[C^{\infty }(M^{A},A)]\text{.}
\end{equation*}

The map%
\begin{equation*}
\mathfrak{X}(M^{A})\times \mathfrak{X}(M^{A})\longrightarrow \mathfrak{X}%
(M^{A}),(X,Y)\longmapsto \lbrack X,Y]=X\circ Y-Y\circ X
\end{equation*}%
is skew-symmetric $A$-bilinear and defines a structure of an $A$-Lie algebra
over $\mathfrak{X}(M^{A})$.

If 
\begin{equation*}
\theta :C^{\infty }(M)\longrightarrow C^{\infty }(M),
\end{equation*}%
is a vector field on $M$, then there exists one and only one $A$-linear
derivation, 
\begin{equation*}
\theta ^{A}:C^{\infty }(M^{A},A)\longrightarrow C^{\infty }(M^{A},A)
\end{equation*}
called prolongation of the vector field $\theta $, such that 
\begin{equation*}
\theta ^{A}(f^{A})=[\theta (f)]^{A}\text{,}\quad \text{for any }\,f\in
C^{\infty }(M)\text{.}
\end{equation*}

If $\theta ,\theta _{1}$ and $\theta _{2}$ are vector fields on $M$ and if $%
f\in C^{\infty }(M)$, then we have:

\begin{enumerate}
\item $(\theta _{1}+\theta _{2})^{A}=\theta _{1}^{A}+\theta _{2}^{A}$;

\item $(f\cdot \theta )^{A}=f^{A}\cdot \theta ^{A}$;

\item $[\theta _{1},\theta _{2}]^{A}=[\theta _{1}^{A},\theta _{2}^{A}]$.
\end{enumerate}

The map%
\begin{equation*}
\mathfrak{X}(M)\longrightarrow Der_{A}[C^{\infty }(M^{A},A)],\theta
\longmapsto \theta ^{A}
\end{equation*}%
is an injective morphism of $\mathbb{R}$-Lie algebras.

\bigskip

\section{Hamiltonian vector fields on weil bundles}

\subsection{Structure of $A$-Poisson manifold on $M^{A}$ when $M$ is a
Poisson manifold}

We recall that a Poisson structure on a smooth manifold $M$ is due to the
existence of a bracket $\left\{ ,\right\} $ on $C^{\infty }(M)$ such that
the pair $(C^{\infty }(M),\{,\})$ is a real Lie algebra such that, for any $%
f\in C^{\infty }(M)$ the map%
\begin{equation}
ad(f):C^{\infty }(M)\longrightarrow C^{\infty }(M),g\longmapsto \{f,g\} 
\notag
\end{equation}%
is a derivation of commutative algebra i.e 
\begin{equation}
\{f,g\cdot h\}=\{f,g\}\cdot h+g\cdot \{f,h\}  \notag
\end{equation}%
for $f,g,h\in C^{\infty }(M)$. In this case we say that $M$ is a Poisson
manifold and $C^{\infty }(M)$ is a Poisson algebra \cite{vai},\cite{vai2}..

We denote%
\begin{equation}
C^{\infty }(M)\longrightarrow Der_{\mathbb{R}}[C^{\infty }(M)],f\longmapsto
ad(f)\text{,}  \notag
\end{equation}%
the adjoint representation and $d_{ad}$ the operator of cohomology
associated to this representation. For any $p\in \mathbb{N}$, 
\begin{equation}
\Lambda _{Pois}^{p}(M)=\mathcal{L}_{sks}^{p}[C^{\infty }(M),C^{\infty }(M)] 
\notag
\end{equation}%
denotes the $C^{\infty }(M)$-module of skew-symmetric multilinear forms of
degree $p$ from $C^{\infty }(M)$ into $C^{\infty }(M)$. We have 
\begin{equation}
\Lambda _{Pois}^{0}(M)=C^{\infty }(M).  \notag
\end{equation}

When $M$ is a smooth manifold, $A$ a weil algebra and $M^{A}$\ the
associated Weil bundle, the $A$-algebra $C^{\infty }(M^{A},A)$ is a Poisson
algebra over $A$ if there exists a bracket $\{,\}$ on $C^{\infty }(M^{A},A)$
such that the pair $(C^{\infty }(M^{A},A),\{,\})$ is a Lie algebra over $A$
satisfying%
\begin{equation*}
\{\varphi _{1}\cdot \varphi _{2},\varphi _{3}\}=\{\varphi _{1},\varphi
_{3}\}\cdot \varphi _{2}+\varphi _{1}\cdot \{\varphi _{2},\varphi _{3}\}
\end{equation*}%
for any $\varphi _{1},\varphi _{2},\varphi _{3}\in C^{\infty }(M^{A},A)$ 
\cite{bok1}..$\newline
$When $M$ is a Poisson manifold with bracket $\{,\}$, for any $f\in
C^{\infty }(M)$, let%
\begin{equation}
\lbrack ad(f)]^{A}:C^{\infty }(M)\longrightarrow C^{\infty
}(M^{A},A),g\longmapsto \{f,g\}^{A}\text{,}  \notag
\end{equation}%
be the prolongation of the vector field $ad(f)$ and let%
\begin{equation}
\widetilde{\lbrack ad(f)]^{A}}:C^{\infty }(M^{A},A)\longrightarrow C^{\infty
}(M^{A},A)  \notag
\end{equation}%
be the unique $A$-linear derivation such that%
\begin{equation*}
\widetilde{\lbrack ad(f)]^{A}}(g^{A})=[ad(f)]^{A}(g)=\{f,g\}^{A}
\end{equation*}%
for any $g\in C^{\infty }(M)$.

For $\varphi \in C^{\infty }(M^{A},A)$, the application%
\begin{equation*}
\tau _{\varphi }:C^{\infty }(M)\longrightarrow C^{\infty
}(M^{A},A),f\longmapsto -\widetilde{[ad(f)]^{A}}(\varphi )
\end{equation*}%
is a vector field on $M^{A}$ considered as derivation of $C^{\infty }(M)$
into $C^{\infty }(M^{A},A)$ and%
\begin{equation*}
\widetilde{\tau _{\varphi }}:C^{\infty }(M^{A},A)\longrightarrow C^{\infty
}(M^{A},A)
\end{equation*}%
the unique $A$-linear derivation (vector field) such that%
\begin{equation}
\widetilde{\tau _{\varphi }}(f^{A})=\tau _{\varphi }(f)=-\widetilde{%
[ad(f)]^{A}}(\varphi )  \notag
\end{equation}%
for any $f\in C^{\infty }(M)$. We have for $f\in C^{\infty }(M)$, 
\begin{equation*}
\widetilde{\tau _{f^{A}}}=\widetilde{[ad(f)]^{A}}\text{,}
\end{equation*}%
and for $\varphi ,\psi \in C^{\infty }(M^{A},A)$ and for $a\in A,$ 
\begin{equation}
\widetilde{\tau }_{\varphi +\psi }=\widetilde{\tau }_{\varphi }+\widetilde{%
\tau }_{\psi };\widetilde{\tau }_{a\cdot \varphi }=a\cdot \widetilde{\tau }%
_{\varphi };\widetilde{\tau }_{\varphi \cdot \psi }=\varphi \cdot \widetilde{%
\tau }_{\psi }+\psi \cdot \widetilde{\tau }_{\varphi }\text{.}  \notag
\end{equation}%
For any $\varphi ,\psi \in C^{\infty }(M^{A},A)$, we let%
\begin{equation}
\{\varphi ,\psi \}_{A}=\widetilde{\tau }_{\varphi }(\psi ).  \notag
\end{equation}%
In \cite{bok1}., we have show that this bracket defines a structure of $A$%
-Poisson algebra on $C^{\infty }(M^{A},A)$.

Thus, when $M$ is a Poisson manifold with bracket $\{,\}$, then $\{,\}_{A}$
is the prolongation on $M^{A}$ of the structure of Poisson on $M$ defined by 
$\{,\}$.

The map%
\begin{equation}
C^{\infty }(M^{A},A)\longrightarrow Der_{A}[C^{\infty }(M^{A},A)],\varphi
\longmapsto \widetilde{\tau _{\varphi }},  \notag
\end{equation}%
is a representation from $C^{\infty }(M^{A},A)$ into $C^{\infty }(M^{A},A).$
We denote $\widetilde{d_{A}}$ the cohomology operator associated to this
adjoint representation \cite{nbo}..\newline
For any $p\in \mathbb{N}$, $\Lambda _{Pois}^{p}(M^{A},\sim _{A})=\mathcal{L}%
_{sks}^{p}[C^{\infty }(M^{A},A),C^{\infty }(M^{A},A)]$ denotes the $%
C^{\infty }(M^{A},A)$-module of skew-symmetric multilinear forms of degree $p
$ on $C^{\infty }(M^{A},A)$ into $C^{\infty }(M^{A},A)$. We have%
\begin{equation}
\Lambda _{Pois}^{0}(M^{A},\sim _{A})=C^{\infty }(M^{A},A)\text{.}  \notag
\end{equation}%
We denote%
\begin{equation}
\Lambda _{Pois}(M^{A},\sim _{A})=\bigoplus_{p=0}^{n}\Lambda
_{Pois}^{p}(M^{A},\sim _{A}).  \notag
\end{equation}%
For $\Omega \in \Lambda _{Pois}^{p}(M^{A},\sim _{A})$ and $\varphi
_{1},\varphi _{2},...,\varphi _{p+1}\in C^{\infty }(M^{A},A)$, we have%
\begin{align*}
\widetilde{d_{A}}\Omega (\varphi _{1},...,\varphi _{p+1})&
=\sum_{i=1}^{p+1}(-1)^{i-1}\widetilde{\tau _{\varphi _{i}}}[\Omega (\varphi
_{1},...,\widehat{\varphi _{i}},...,\varphi _{p+1}] \\
& +\sum_{1\leq i<j\leq p+1}(-1)^{i+j}\Omega (\{\varphi _{i},\varphi
_{j}\}_{A},\varphi _{1},...,\widehat{\varphi _{i}},...,\widehat{\varphi _{j}}%
,...,\varphi _{p+1})
\end{align*}%
where $\widehat{\varphi _{i}}$ means that the term $,\varphi _{i}$ is
omitted.

\begin{proposition}
For any $\eta \in \Lambda _{Pois}^{p}(M)$, we have 
\begin{equation*}
\widetilde{d_{A}}(\eta ^{A})=(d\eta )^{A}\text{.}
\end{equation*}
\end{proposition}

\section{Hamiltonian vector fields on weil bundles}

When $M$ is a Poisson manifold with bracket $\left\{ ,\right\} $, a vector
field 
\begin{equation*}
\theta :C^{\infty }(M)\longrightarrow C^{\infty }(M)
\end{equation*}

\begin{enumerate}
\item is locally hamiltonian if $\theta $ is closed for the cohomology
associated with the adjoint representation%
\begin{equation*}
ad:C^{\infty }(M)\longrightarrow Der\left[ C^{\infty }(M)\right]
\end{equation*}%
i.e. $d_{ad}\theta =0$.

\item is globally hamiltonian if $\theta $ is exact for the cohomology
associated with the adjoint representation%
\begin{equation*}
ad:C^{\infty }(M)\longrightarrow Der\left[ C^{\infty }(M)\right]
\end{equation*}%
i.e. there exists $f\in C^{\infty }(M)$ such that $\theta =d_{ad}(f)$.
\end{enumerate}

Thus, a vector field 
\begin{equation*}
X:C^{\infty }(M^{A},A)\longrightarrow C^{\infty }(M^{A},A)
\end{equation*}

\begin{enumerate}
\item is locally hamiltonian if $X$ is closed for the cohomology associated
with the adjoint representation%
\begin{equation*}
\widetilde{\tau _{\varphi }}:C^{\infty }(M^{A},A)\longrightarrow Der_{A}%
\left[ C^{\infty }(M^{A},A)\right]
\end{equation*}%
i.e $\widetilde{d_{A}}X=0$.

\item is globally hamiltonian if $X$ is exact for the cohomology associated
with the adjoint representation%
\begin{equation*}
\widetilde{\tau _{\varphi }}:C^{\infty }(M^{A},A)\longrightarrow Der_{A}%
\left[ C^{\infty }(M^{A},A)\right]
\end{equation*}%
i.e. there exists $\varphi \in C^{\infty }(M^{A},A)$ such that $X=\widetilde{%
d_{A}}(\varphi )$.
\end{enumerate}

\begin{proposition}
When $M$ is a Poisson manifold with bracket $\left\{ ,\right\} $, then a
vector field%
\begin{equation*}
\theta :C^{\infty }(M)\longrightarrow C^{\infty }(M)
\end{equation*}%
is locally hamiltonian if and only if \ the vector field%
\begin{equation*}
\theta ^{A}:C^{\infty }(M^{A},A)\longrightarrow C^{\infty }(M^{A},A).
\end{equation*}%
is locally hamiltonian.
\end{proposition}

\begin{proof}
Indeed, for any $\eta \in \Lambda _{Pois}^{p}(M)$, we have 
\begin{equation*}
\widetilde{d_{A}}(\eta ^{A})=(d_{ad}\eta )^{A}.
\end{equation*}%
In particular, for $p=1,$ we have%
\begin{equation*}
\widetilde{d_{A}}(\theta ^{A})=(d_{ad}\theta )^{A}\text{.}
\end{equation*}%
Thus, $d_{ad}\theta =0$ if and only if $\widetilde{d_{A}}(\theta ^{A})=0$.
\end{proof}

\bigskip

\begin{proposition}
When $M^{A}$ is a $A$-Poisson manifold with bracket $\left\{ ,\right\} _{A}$%
, then, a vector field 
\begin{equation*}
X:C^{\infty }(M^{A},A)\longrightarrow C^{\infty }(M^{A},A)
\end{equation*}%
locally hamiltonian is a derivation of the Poisson $A$-algebra $C^{\infty
}(M^{A},A).$
\end{proposition}

\begin{proof}
We have%
\begin{equation*}
\begin{array}{cccc}
\widetilde{d_{A}}X: & C^{\infty }(M^{A},A)\times C^{\infty }(M^{A},A) & 
\longrightarrow  & C^{\infty }(M^{A},A) \\ 
& (\varphi ,\psi ) & \longmapsto  & (\widetilde{d_{A}}X)(\varphi ,\psi )%
\end{array}%
\end{equation*}%
and if $\widetilde{d_{A}}X=0$, then for any $\varphi ,\psi \in C^{\infty
}(M^{A},A)$,%
\begin{eqnarray*}
0 &=&(\widetilde{d_{A}}X)(\varphi ,\psi ) \\
&=&\widetilde{\tau _{\varphi }}[X(\psi )]-\widetilde{\tau _{\psi }}[X(\psi
)]-X(\left\{ \varphi ,\psi \right\} _{A}) \\
&=&\{\varphi ,X(\psi )\}_{A}-\{\psi ,X(\varphi )\}_{A}-X(\{\varphi ,\psi
\}_{A})\text{ }
\end{eqnarray*}%
i.e%
\begin{equation*}
X(\left\{ \varphi ,\psi \right\} _{A})=\{X(\varphi ),\psi \}_{A}+\{\varphi
,X(\psi )\}_{A}\text{.}
\end{equation*}%
That ends the proof.
\end{proof}

\begin{proposition}
Let $M$ be a Poisson manifold with bracket $\left\{ ,\right\} $. If a vector
field 
\begin{equation*}
\theta :C^{\infty }(M)\longrightarrow C^{\infty }(M)
\end{equation*}%
is globally hamiltonian then\ the vector field 
\begin{equation*}
\theta ^{A}:C^{\infty }(M^{A},A)\longrightarrow C^{\infty }(M^{A},A)
\end{equation*}%
is globally hamiltonian.
\end{proposition}

\begin{proof}
Based on the assumptions, there exists $f\in C^{\infty }(M)$ such that $%
\theta =d_{ad}(f)$. Thus, 
\begin{eqnarray*}
\theta ^{A} &=&[ad(f)]^{A} \\
&=&\widetilde{d_{A}}(f^{A}).
\end{eqnarray*}%
Thus, $\theta =d_{ad}(f)$ then $\theta ^{A}=\widetilde{d_{A}}(f^{A})$ is
globally hamiltonian.
\end{proof}

\begin{proposition}
When $M^{A}$ is a $A$-Poisson manifold with bracket $\left\{ ,\right\} _{A}$%
, then a vector field%
\begin{equation*}
X:C^{\infty }(M^{A},A)\longrightarrow C^{\infty }(M^{A},A)
\end{equation*}%
globally hamiltonian is the derivation interior of the Poisson $A$-algebra $%
C^{\infty }(M^{A},A)$.
\end{proposition}

\begin{proof}
If the vector field 
\begin{equation*}
X:C^{\infty }(M^{A},A)\longrightarrow C^{\infty }(M^{A},A)
\end{equation*}%
is globally hamiltonian, there exists $\varphi \in C^{\infty }(M^{A},A)$
suth that $X=\widetilde{d_{A}}\varphi .$ For any $\psi \in C^{\infty
}(M^{A},A),$ we have%
\begin{eqnarray*}
X(\psi ) &=&(\widetilde{d_{A}}\varphi )(\psi ) \\
&=&\widetilde{\tau }_{\varphi }(\psi ) \\
&=&\{\varphi ,\psi \}_{A}
\end{eqnarray*}%
i.e. $X=ad(\varphi )$. $\ $where 
\begin{equation*}
ad(\varphi ):C^{\infty }(M^{A},A)\longrightarrow C^{\infty }(M^{A},A),\psi
\longmapsto \{\varphi ,\psi \}_{A}
\end{equation*}

Thus, $X$ is globally hamiltonian if there exists $\varphi \in C^{\infty
}(M^{A},A)$ such that $X=\widetilde{\tau _{\varphi }}=ad(\varphi )$ i.e. $X$
is the interior derivation of the Poisson $A$-algebra $C^{\infty }(M^{A},A).$
\end{proof}

\bigskip

\subsection{\protect\bigskip Hamiltonian vector fields on $M^{A}$ $\ $when $%
M $ is a symplectic manifold}

When $(M,\Omega )$ is a symplectic manifold, then $(M^{A},\Omega ^{A})$ is a
symplectic $A$-manifold \cite{bok1}.

For any $f\in C^{\infty }(M)$, we denote $X_{f}$ the unique vector field on $%
M$ \ such that 
\begin{equation*}
i_{X_{f}}\Omega =df
\end{equation*}%
where 
\begin{equation*}
d:\Lambda (M)\longrightarrow \Lambda (M)
\end{equation*}%
is the operator of de Rham cohomology. We denote 
\begin{equation*}
d^{A}:\Lambda (M^{A},A)\longrightarrow \Lambda (M^{A},A)
\end{equation*}%
the operator of cohomology associated with the representation 
\begin{equation*}
\mathfrak{X}(M^{A})\longrightarrow Der\left[ C^{\infty }(M^{A},A)\right]
,X\longmapsto X\text{.}
\end{equation*}%
For $\varphi \in C^{\infty }(M^{A},A)$, we denote $X_{\varphi }$ the unique
vector field on $M^{A}$, considered as a derivation of $C^{\infty }(M^{A},A)$
into $C^{\infty }(M^{A},A)$, such that 
\begin{equation*}
i_{X_{\varphi }}\Omega ^{A}=d^{A}(\varphi )\text{.}
\end{equation*}%
The bracket 
\begin{align*}
\left\{ \varphi ,\psi \right\} _{\Omega ^{A}}& =-\Omega ^{A}(X_{\varphi
},X_{\psi }) \\
& =X_{\varphi }(\psi )
\end{align*}%
defines a structure of $A$-Poisson manifold on $M^{A}$ and for any $f\in
C^{\infty }(M)$, 
\begin{equation*}
X_{f^{A}}=(X_{f})^{A}\text{.}
\end{equation*}

\begin{equation*}
i_{(X_{f})^{A}}\Omega ^{A}=i_{X_{f^{A}}}\Omega ^{A}\text{.}
\end{equation*}

We deduce that \cite{bok1}.:

\begin{theorem}
If $(M,\Omega )$ is a symplectic manifold, the structure of $A$-Poisson
manifold on $M^{A}$ defined by $\Omega ^{A}$ coincide with the prolongation
on $M^{A}$ of the Poisson structure on $M$ defined by the symplectic form $%
\Omega $ i.e for any $\varphi \in C^{\infty }(M^{A},A)$, $\widetilde{\tau
_{\varphi }}=X_{\varphi }$.
\end{theorem}

Therefore, for any $\varphi ,\psi \in C^{\infty }(M^{A},A)$, we have 
\begin{equation*}
\left\{ \varphi ,\psi \right\} _{\Omega ^{A}}=\left\{ \varphi ,\psi \right\}
_{A}\text{.}
\end{equation*}

\begin{proposition}
If $\omega $ is a differential form on $M$ and if $\theta $ is a vector
field on $M$, then 
\begin{equation*}
(i_{\theta }\omega )^{A}=i_{\theta ^{A}}(\omega ^{A})\text{.}
\end{equation*}
\end{proposition}

\begin{proof}
If the degree of $\omega $ is $p$, then $(i_{\theta }\omega )^{A}$ is the
unique differential $A$-form of degree $p-1$ such that 
\begin{align*}
(i_{\theta }\omega )^{A}(\theta _{1}^{A},...,\theta _{p-1}^{A})& =\left[
(i_{\theta }\omega )(\theta _{1},...,\theta _{p-1})\right] ^{A} \\
& =\left[ \omega (\theta ,\theta _{1},...,\theta _{p-1})\right] ^{A}
\end{align*}%
for any $\theta _{1},\theta _{2},...,\theta _{p-1}\ \in $ $\mathfrak{X}(M)$.
As $i_{\theta ^{A}}(\omega ^{A})$ is of degree $p-1$ and is such that%
\begin{align*}
i_{\theta ^{A}}(\omega ^{A})\left[ \theta _{1}^{A},...,\theta _{p-1}^{A}%
\right] & =\omega ^{A}(\theta ^{A},\theta _{1}^{A},...,\theta _{p-1}^{A}) \\
& =\left[ \omega (\theta ,\theta _{1},...,\theta _{p-1})\right] ^{A}
\end{align*}%
for any $\theta _{1},\theta _{2},...,\theta _{p-1}\ \in $ $\mathfrak{X}(M)$,
we conclude that $(i_{\theta }\omega )^{A}=i_{\theta ^{A}}(\omega ^{A})$.
\end{proof}

When $(M,\Omega )$ is a symplectic manifold,

\begin{enumerate}
\item a vector field $\theta $ on $M$ is locally hamiltonian if the form $%
i_{\theta }\Omega $ is closed for the de Rham cohomology and $\theta $ is
globally hamiltonian if there exists $f\in C^{\infty }(M)$ such that $%
i_{\theta }\Omega =-d(f),$ i.e. the form $i_{\theta }\Omega $ is $d$-exact.

\item a vector field $X$ on $M^{A}$ is locally hamiltonian if the form $%
i_{X}\Omega ^{A}$ is $d^{A}$-closed and $X$ is globally hamiltonian if there
exists $\varphi \in C^{\infty }(M^{A},A)$ such that $i_{X}\Omega
^{A}=-d^{A}(\varphi )$, i.e. the form $i_{X}\Omega ^{A}$ is $d^{A}$-exact.
\end{enumerate}

\begin{proposition}
A vector field $\theta :C^{\infty }(M)\longrightarrow C^{\infty }(M)$ on a
symplectic manifold $M$ is locally hamiltonian, if and only if $\theta
^{A}:C^{\infty }(M^{A},A)\longrightarrow C^{\infty }(M^{A},A)$ is a locally
hamiltonian vector field.
\end{proposition}

\begin{proof}
For any $\theta \in \mathfrak{X}(M)$, we have 
\begin{eqnarray*}
d^{A}(i_{\theta ^{A}}\Omega ^{A}) &=&d^{A}[(i_{\theta }\Omega )^{A}] \\
&=&[d(i_{\theta }\Omega )]^{A}.
\end{eqnarray*}%
Thus, $\theta $ is locally hamiltonian, i.e $d(i_{\theta }\Omega )=0$ if and
only if, $d^{A}(i_{\theta }\Omega ^{A})=0$ i.e $\theta ^{A}:C^{\infty
}(M^{A},A)\longrightarrow C^{\infty }(M^{A},A)$ is a locally hamiltonian
vector field.
\end{proof}

\begin{theorem}
A vector field $X:C^{\infty }(M^{A},A)\longrightarrow C^{\infty }(M^{A},A)$
on $M^{A}$ locally hamiltonian is a derivation of the $A$-Lie algebra
induced by the $A$-structure of Poisson defined by the symplectic $A$%
-manifold $(M^{A},\Omega ^{A}).$
\end{theorem}

\begin{proof}
Let $(M^{A},\Omega ^{A})$ be a symplectic manifold. For any $\varphi ,\psi
\in C^{\infty }(M^{A},A)$,%
\begin{align*}
\left\{ \varphi ,\psi \right\} _{\Omega ^{A}}& =-\Omega ^{A}(X_{\varphi
},X_{\psi }) \\
& =X_{\varphi }(\psi )
\end{align*}%
If $X$ is locally hamiltonian vector field, we have $d^{A}(i_{X}\Omega
^{A})=0$ i.e. for any $X$ and $Y\in \mathfrak{X}(M^{A})$,%
\begin{equation*}
d^{A}(i_{X}\Omega ^{A})(X,Y)=0\text{.}
\end{equation*}%
In particular, for any $X_{\varphi }$ and $X_{\psi },$ we have $\ $%
\begin{eqnarray*}
0 &=&(d^{A}(i_{X}\Omega ^{A}))(X_{\varphi },X_{\psi }) \\
&=&X_{\varphi }[i_{X}\Omega ^{A}(X_{\psi })]-X_{\psi }[i_{X}\Omega
^{A}(X_{\varphi })]-i_{X}\Omega ^{A}([X_{\varphi },X_{\psi }])
\end{eqnarray*}

Therefore%
\begin{equation*}
i_{X}\Omega ^{A}([X_{\varphi },X_{\psi }])=X_{\varphi }[i_{X}\Omega
^{A}(X_{\psi })]-X_{\psi }[i_{X}\Omega ^{A}(X_{\varphi })]
\end{equation*}%
i.e%
\begin{equation*}
\Omega ^{A}(X,[X_{\varphi },X_{\psi }])=X_{\varphi }[\Omega ^{A}(X,X_{\psi
})]-X_{\psi }[\Omega ^{A}(X,X_{\varphi })]\text{ }
\end{equation*}

Hence%
\begin{equation*}
X(\{\varphi ,\psi \}_{\Omega ^{A}})=\{X(\varphi ),\psi \}_{\Omega
^{A}}+\{\varphi ,X(\psi )\}_{\Omega ^{A}}.\text{ }
\end{equation*}%
That ends the proof.
\end{proof}

\begin{proposition}
Let $(M,\Omega )$  be a symplectic manifold. If a vector field%
\begin{equation*}
\theta :C^{\infty }(M)\longrightarrow C^{\infty }(M)
\end{equation*}%
is globally hamiltonian then\ the vector field 
\begin{equation*}
\theta ^{A}:C^{\infty }(M^{A},A)\longrightarrow C^{\infty }(M^{A},A)
\end{equation*}%
is globally hamiltonian.
\end{proposition}

\begin{proof}
If $\theta $ is globally hamiltonian, then there exists $f\in C^{\infty }(M)$
such that $i_{\theta }\Omega =-d(f)$. Then, 
\begin{eqnarray*}
\left( i_{\theta }\Omega \right) ^{A} &=&[-d(f)]^{A} \\
&=&d^{A}\left( f^{A}\right) 
\end{eqnarray*}

Thus%
\begin{equation*}
i_{\theta ^{A}}\Omega ^{A}=d^{A}\left( f^{A}\right) \text{.}
\end{equation*}
i.e  $\theta ^{A}$ is globally hamiltonian.
\end{proof}

\end{document}